\documentclass[preprint,12pt]{elsarticle}

\usepackage{amssymb}
\usepackage{multirow}

\usepackage{amsthm}
\usepackage{graphics}
\usepackage{epsfig}
\usepackage{amssymb}
\usepackage{graphicx}
\usepackage{subfigure}
\usepackage{color}
\usepackage{tikz}
\usetikzlibrary{patterns}
\usepackage{hyperref}
\usepackage{a4wide}
\usepackage{amsmath}
\usepackage{amsfonts}
\usepackage{enumerate}
\newcommand{\pf}{\noindent {\bf Proof: }}

\newtheorem*{theoremaux}{Theorem \theoremauxnum}
\gdef\theoremauxnum{1}

\newtheorem{lemma}{\bf Lemma}[section]
 
\newtheorem{theorem}{\bf Theorem}[section]
\newtheorem{proposition}[lemma]{\bf Proposition}
\newtheorem{corollary}[lemma]{\bf Corollary}
\newtheorem{definition}{\bf Definition}[section]
\newtheorem{remark}{\bf Remark}[section]

\journal{~}

\begin{document}

\begin{frontmatter}



\title{Cyclic Probability of a Finite Group}




\author{Sayanta Chatterjee\corref{cor1}}
\ead{sayanta.math@gmail.com}

\author{Angsuman Das}
\ead{angsuman.maths@presiuniv.ac.in}
\address{Department of Mathematics, Presidency University, Kolkata, India}

\cortext[cor1]{Corresponding author}

\begin{abstract}
The cyclic probability, $\theta(G)$, of a finite group $G$, is defined as the probability that two randomly selected elements of $G$ generate a cyclic subgroup of $G$. In this paper, we study various upper and lower bounds of $\theta(G)$ and show that $\theta(G)$ can be used as a criterion for checking nilpotency and solvability of a finite group. We also compare $\theta(G)$ with other group invariants like commuting probability $cp(G)$ and normalized sum of element orders. 
\end{abstract}


\end{frontmatter}


\section{Introduction}
The study of probabilistic properties of finite groups has a rich history, dating at least to the work of Gustafson \cite{gustafson} and Gallagher \cite{gallagher} on the commuting probability $$cp(G)=\dfrac{|L(G)|}{|G \times G|}=\dfrac{K(G)}{|G|}.$$ where $L(G)=\{(x,y)\in G\times G: xy=yx\}=\{(x,y)\in G\times G: \langle x,y\rangle \mbox{ is abelian} \}$ and $K(G)$ denote the number of conjugacy classes in $G$.
 This invariant, which measures the chance that two randomly chosen group elements commute, has been shown to encode significant structural information about $G$. For instance, Gustafson proved that $cp(G)\leq 5/8$ for any non-abelian finite group, with equality characterizing groups isoclinic to the dihedral group $D_4$ or the quaternion group $Q_8$. Subsequent work by Rusin \cite{rusin}, Lescot \cite{lescot}, and others has refined these results and classified groups whose commuting probability exceeds various thresholds.

A natural and important refinement of the commuting condition is to require not merely that $x$ and $y$ commute, but that the entire subgroup they generate is cyclic. This leads the authors in \cite{cycprob_1st_paper} to introduce the notion of cyclic probability. 
\begin{definition}\cite{cycprob_1st_paper}
    Let $G$ be a finite group and $S_G=\{(x,y)\in G\times G: \langle x,y\rangle \mbox{ is cyclic.}\}$. The cyclic probability of $G$ is defined by $$\theta(G)=\dfrac{|S_G|}{|G\times G|}.$$
\end{definition}

While $cp(G)$ measures how often two elements commute, $\theta(G)$ asks how often they generate a cyclic subgroup. The latter condition is strictly stronger, and accordingly $\theta(G)\leq cp(G)$ for every finite group $G$. However, the cyclic probability captures subtler aspects of group structure, particularly in characterizing nilpotent, supersolvable and solvable groups with respect to certain thresholds. Surprisingly, to the best of our knowledge, no follow up works after \cite{cycprob_1st_paper} were done to explore these possibilities.

Apart from commuting probability, there exists some other group invariants (see Definition \ref{sum_of_order_defns}) which are also popular in literature for providing alternative sufficient criterion for nilpotency and solvability of finite groups. We recall a few of them and compare how these invariants perform with respect to cyclic probability.

\subsection{Preliminaries}
We first recall some other relevant group invariants existing in literature.
\begin{definition}\label{sum_of_order_defns}
    Let $G$ be a finite group of order $n$. Define\footnote{It was proved in \cite{amiri-1st} that $\sigma(G)< \sigma(\mathbb{Z}_n)$ for any non-cyclic group $G$ of order $n$. Thus $0<\sigma'(G)\leq 1$.} 
    
    $$\cite{amiri-1st}~ \sigma(G)=\sum_{g \in G} o(g), ~~~ \cite{azad-khosravi-2018}~ \sigma'(G)=\dfrac{\sigma(G)}{\sigma(\mathbb{Z}_n)}, ~~~~~\cite{marius-israel} ~\sigma''(G)=\dfrac{\sigma(G)}{|G|^2}.$$

\end{definition}
It is to be noted that $\sigma, \sigma'$ and $\sigma''$ were originally defined in their respective papers as $\psi,\psi'$ and $\psi''$. However, in this paper, as we will be using the Dedekind Psi function $\psi$ as a key-ingredient, we follow this re-nomenclature.

The above invariants have been used to provide alternative sufficient criterion for nilpotency and solvability of finite groups. We summarize those results in Table \ref{sum_of_order_table}.

\begin{table}[h]
    \centering
    \begin{tabular}{||c||c||}
\hline        Threshold for $\sigma'(G)$ & Threshold for $\sigma''(G)$\\ \hline \hline
      &  \\
      \cite{marius-commalg} $\sigma'(G)>\dfrac{13}{21}\Rightarrow G$ is nilpotent  & \cite{marius-israel} $\sigma''(G)>\dfrac{13}{36}\Rightarrow G$ is nilpotent\\ 
      & \\
\hline 
 & \\ 
 \cite{azad-khosravi-2022} $\sigma'(G)>\dfrac{31}{77}\Rightarrow G$ is supersolvable &  \cite{marius-israel} $\sigma''(G)>\dfrac{31}{144}\Rightarrow G$ is supersolvable\\
 & \\
 \hline
 & \\
 \cite{azad-khosravi-2018} $\sigma'(G)>\dfrac{211}{1617}\Rightarrow G$ is solvable & \cite{marius-israel} $\sigma''(G)>\dfrac{211}{3600}\Rightarrow G$ is solvable \\
 & \\ \hline \hline 
    \end{tabular}
    \caption{Sufficient Conditions for Nilpotency and Solvability}
    \label{sum_of_order_table}
\end{table}

Next we recall two well-known multiplicative functions:
\begin{definition}
Jordan's totient function, denoted as $\mathcal{J}_k(n)$, where $k$ is a positive integer, is a function of a positive integer, $n$ that equals the number of $k$-tuples of positive integers $(a_1,a_2,\ldots,a_k)$ such that $a_i\leq n$ for all $i$ and $gcd(a_1,a_2,\ldots,a_k,n)=1$.    
\end{definition}
$\mathcal{J}_k(n)$ is a multiplicative function, which is a generalization of Euler's totient function, with $\varphi(n)=\mathcal{J}_1(n)$ and it is given by $$\mathcal{J}_k(n)=n^k\prod_{p\mid n}\left(1-\dfrac{1}{p^k} \right), \mbox{ where }p \mbox{ runs over prime divisors of }n.$$
\begin{definition}
    The Dedekind psi function is another multiplicative function defined by $$\psi(1)=1\mbox{ \& }\psi(n)=n\prod_{p\mid n}\left(1+\dfrac{1}{p} \right), \mbox{ where }p \mbox{ runs over  prime divisors of }n.$$
\end{definition}
It can be shown that $$\psi(n)=\dfrac{\mathcal{J}_2(n)}{\mathcal{J}_1(n)}=\dfrac{\mathcal{J}_2(n)}{\varphi(n)}.$$

Some other well-known properties of $\psi(n)$ are given by the following lemma.
\begin{lemma}\label{properties_of_psi}
    The following are true:
    \begin{enumerate}
        \item $\psi$ is multiplicative, i.e., $\psi(ab)=\psi(a)\cdot \psi(b)$ if $gcd(a,b)=1$.
        \item $\psi$ is submultiplicative, i.e., $\psi(ab)\leq \psi(a)\cdot \psi(b)$.
        \item If $a\mid b$, then $\psi(a)\leq \psi(b)$.
        \item If $a\mid b$ and $\psi(a)= \psi(b)$, then $a=b$.  
    \end{enumerate}
\end{lemma}

\subsection{Our Contribution}
In Section \ref{general_bounds_section}, we prove some upper and lower bounds of $\theta(G)$ in terms of other group invariants and quotients. In Section \ref{specific_values_section}, we classify groups for specific values of $\theta(G)$ like $5/8,1/2,7/16$ etc. In Section \ref{sufficient_conditions_section}, we provide some sufficient conditions for supersolvability and solvability of a group $G$ based on the value of $\theta(G)$.  Moreover, we compare these results with those mentioned in Table \ref{sum_of_order_table}. 


\section{General Bounds}\label{general_bounds_section}
Before delving into specific values of the cyclic probability, it is natural to ask how $\theta(G)$ relates to other group-theoretic invariants and how it behaves under basic constructions such as quotients and direct products. In this section, we establish general upper and lower bounds for $\theta(G)$ in terms of the commuting probability, the sum of element orders, and the group’s exponent. We also derive explicit bounds depending on the smallest prime divisor of $|G|$ and characterize when equality can occur. We start by noting the values of $\theta(G)$ for some well known families of groups (see Table \ref{family}), which can be computed using Proposition \ref{psi_connection_proposition}.

\begin{table}[h]
        \centering
        \begin{tabular}{|c|c|c||c|c|c|}
          \hline  Group $G$ & Order & $\theta(G)$ & Group $G$ & Order & $\theta(G)$\\ \hline 
           & & & & &\\
           $\mathbb{Z}^n_p$ & $p^n, ~(n\geq 1)$ & $\dfrac{p^{n}+p^{n-1}-1}{p^{2n-1}}$ & $\mathbb{Z}_{2^{n-1}}\times \mathbb{Z}_2$  & $2^n, ~(n\geq 2)$ & $\dfrac{1}{2}+\dfrac{1}{2^{2n-1}}$\\
           & & & & &\\ \hline 
            & & & & &\\
           $D_n$  & $2n, ~(n\geq 3)$ & $\dfrac{1}{4}+\dfrac{3}{4n}$ & $Q_{2^n}$  & $2^n,  ~(n\geq 3)$ & $\dfrac{1}{4}+\dfrac{3}{2^n}$\\
           & & & & &\\ \hline 
            & & & & &\\
           $SD_{2^n}$  & $2^n,  ~(n\geq 4)$ & $\dfrac{1}{4}+\dfrac{9}{2^{n+2}}$ & $M_{2^n}$  & $2^n,  ~(n\geq 4)$ & $\dfrac{1}{2}+\dfrac{1}{2^{2n-1}}$\\ 
           & & & & &\\ \hline 
        \end{tabular}
        \caption{Cyclic Probability of Some Families of Groups}
        \label{family}
\end{table}

\begin{proposition}
    Let $G$ be a finite group. Then $cp(G)=\theta(G)$ if and only if all Sylow subgroups of $G$ are either cyclic or generalized quaternion.
\end{proposition}
\pf It follows from the fact that $S_G\subseteq L_G$ and Theorem 30 \cite{enhanced-1st paper}.

\begin{proposition}\label{psi_connection_proposition}
Let $G$ be a finite group, $\mathcal{C}(G)$ be the set of all cyclic subgroups of $G$ and $c_k$ be the number of elements of order $k$ in $G$. Then $$|S_G|=\sum_{C \in \mathcal{C}(G)} \mathcal{J}_2(|C|)=\sum_{g \in G}\psi(o(g))=\sum_{k\mid |G|}c_k\psi(k)\leq K(G)\cdot |G|.$$
\end{proposition}
\pf For a given cyclic subgroup $C$ of $G$, let $f(C)$ be the number of ordered pairs $(x,y)\in C \times C$ such that $\langle x,y \rangle=C$. Then $|S_G|=\sum_{C \in \mathcal{C}(G)}f(C)$. So, it is enough to find $f(C)$. Let $|C|=d$. Then $C \cong \mathbb{Z}_d$. Now $\langle x,y \rangle=\mathbb{Z}_d$ if and only if $gcd(x,y,d)=1$. Thus $f(C)=\mathcal{J}_2(|C|)$. Hence the first equality follows.

As every cyclic subgroup $C$ of order $d$ has exactly $\varphi(d)$ generators, we distribute $\mathcal{J}_2(|C|)$ equally among all the generators of $C$. Thus we have $$|S_G|=\sum_{g \in G}\dfrac{\mathcal{J}_2(o(g))}{\varphi(o(g))}=\sum_{g \in G}\psi(o(g)).$$

The third equality is obtained by suitable grouping of terms in the second equality. 

The last inequality follows from the formula of commuting probability $cp(G)$.\qed

\begin{proposition}\label{quotient-direct-product-prop}
    Let $G$ and $H$ be finite groups and $N\lhd G$. Then the following are true:
    \begin{itemize}
        \item $\theta(G)\leq \theta(G/N)$.
        \item $\theta(G\times H)\leq \theta(G)\cdot \theta(H)$. Equality holds if and only if $gcd(|G|,|H|)=1$.
    \end{itemize}
\end{proposition}

\pf Let $S_G=\{(x,y)\in G \times G: \langle x,y \rangle \mbox{ is cyclic}\}$ and $S_{G/N}=\{(aN,bN)\in G/N \times G/N: \langle aN,bN \rangle \mbox{ is cyclic}\}$. Define $\pi: G \rightarrow G/N$ by $\pi(g)=gN$. Then $(x,y)\in S_G$ implies $(\pi(x),\pi(y))\in S_{G/N}$. Thus $$S_G\subseteq \pi^{-1}(S_{G/N})=\{(x,y)\in G\times G: \langle \pi(x),\pi(y) \rangle \mbox{ is cyclic}\}.$$

For a fixed $(aN,bN)\in S_{G/N}$, the set of all $(x,y)$ with $\pi(x)=aN$ and $\pi(y)=bN$ is precisely the coset $aN\times bN$, which has size $|N|^2$. Hence $$|\pi^{-1}(S_{G/N})|=|S_{G/N}|\cdot |N|^2, \mbox{ i.e., }|S_G|\leq |S_{G/N}|\cdot |N|^2.$$ Dividing both the sides by $|G|^2=|G/N|^2\cdot |N|^2$, we get $\theta(G)\leq \theta(G/N)$.

For the next inequality, $$|S_{G\times H}|=\sum_{(g,h)\in G\times H} \psi(o(g,h))=\sum_{(g,h)\in G\times H} \psi(lcm(o(g),o(h)))\leq \sum_{(g,h)\in G\times H} \psi(o(g)\cdot o(h))$$

$$\leq \sum_{(g,h)\in G\times H} \psi(o(g))\cdot \psi(o(h))=\left(\sum_{g \in G} \psi(o(g)\right)\cdot \left(\sum_{h \in H} \psi(o(h)\right)=|S_G|\cdot |S_H|.$$
Dividing both the sides by $|G\times H|^2=|G|^2\cdot |H|^2$, we get the required inequality. Using Lemma \ref{properties_of_psi}(4), it follows that equality holds if and only if $lcm(o(g),o(h))=o(g)\cdot o(h)$ for all $g \in G, h \in H$ if and only if $gcd(|G|,|H|)=1$.



\begin{remark}
     Unlike $cp(G)$, $\theta(G)$ does not satisfy the following inequalities, in general: 
     \begin{enumerate}
        \item $\theta(G)\leq \theta(N)$.
         \item $\theta(G)\leq \theta(N)\cdot \theta(G/N)$.
     \end{enumerate}
      For example, if $G$ is a group of order $96$ with GAP id $(96,30)$, then it has a normal subgroup $N\cong \mathbb{Z}_2\times D_{12}$ of order $48$ and $G/N\cong \mathbb{Z}_2$ such that $$\theta(N)=\dfrac{21}{128}<\dfrac{85}{512}=\theta(G)$$
\end{remark}

\begin{theorem}\label{exponent-upper-bound}
    For any finite group $G$, $$\sigma''(G)\leq \theta(G)\leq \dfrac{\psi(exp(G))}{|G|}.$$
\end{theorem}
\pf As $a|b$ implies $\psi(a)\leq \psi(b)$, the upper bound follows from the fact that $$|S_G|=\sum_{g \in G}\psi(o(g))\leq \sum_{g \in G}\psi(exp(G))=|G|\cdot \psi(exp(G)).$$ As $o(g)\leq \psi(o(g))$ for all $g \in G$, it follows that $\sigma(G)\leq |S_G|$. Dividing both sides by $|G|^2$, we get $\sigma''(G)\leq \theta(G)$. \qed 

\begin{remark}
    Note that $\sigma'(G)$ and $\theta(G)$ are not comparable in general. For example, if $G\cong (((\mathbb{Z}_9\times \mathbb{Z}_3)\rtimes \mathbb{Z}_3)\rtimes \mathbb{Z}_2$ (GAP id: $(162,14)$), then $221/2187=\theta(G)>\sigma'(G)=1447/14763$, whereas for all groups $G$ of order $<162$, $\theta(G)\leq \sigma'(G)$ holds.
\end{remark}

\begin{theorem}\label{smallest-prime-theorem}
    If $G$ is a finite non-cyclic group and $p$ is the smallest prime dividing $|G|$, then $$\dfrac{(p+1)|G|-p}{|G|^2} \leq \theta(G)\leq \dfrac{p^2+p-1}{p^3}.$$ Moreover, equality in the upper bound occurs if and only if $G\cong Q_8\times \mathbb{Z}_m$ (where $m$ is odd) or $(\mathbb{Z}_p\times\mathbb{Z}_p) \times \mathbb{Z}_m$ (where $m$ is such that all prime factors of $m$ are greater than $p$).
\end{theorem}
\pf Since $\psi(n)$ is strictly increasing for divisors, the minimum value $\psi(o(g))$ (where $g$ is non-trivial element of $G$) can take is $\psi(p) = p + 1$ and $\psi(o(e))=1$. Thus the lower bound follows from the observation that $$|S_G|=\sum_{g \in G}\psi(o(g))\geq 1+ \sum_{g \in G}(|G|-1)(p+1).$$

For the upper bound, first note that for $p=2$, the formula gives the upper bound $5/8$, which has been proved in Theorem \ref{5/8-bound-theorem}. So, we assume that $p$ is an odd prime. If $G$ is non-abelian, then $\theta(G)\leq cp(G)\leq \frac{p^2+p-1}{p^3}$ (by Theorem 14, \cite{sen}). Moreover, equality holds if and only if $G/Z(G)\cong \mathbb{Z}_p\times \mathbb{Z}_p$, which implies $G$ is nilpotent. So we assume that $G$ is abelian, non-cyclic group of odd order.

Since $G$ is nilpotent, it is the direct product of its Sylow subgroups and at least one Sylow subgroup, say Sylow $q$-subgroup $Q$, is non-cyclic. As $Q$ is non-cyclic, $Q/\Phi(Q)\cong (\mathbb{Z}_q)^k$ with $k\geq 2$, i.e., $$\theta(G)\leq \theta(Q)\leq \theta(Q/\Phi(Q))\leq \theta(\mathbb{Z}^k_q)\leq \dfrac{q^2+q-1}{q^3}.$$ As $p\leq q$ and $x\mapsto \frac{x^2+x-1}{x^3}$ is a strictly decreasing function, we have the required upper bound. It is to be noted that for equality in the upper bound, $G$ must be nilpotent and hence $$\dfrac{p^2+p-1}{p^3}=\theta(G)=\prod_{q\mid |G|} \theta(G_q).$$ This implies that all Sylow subgroups $G_q$'s, except the one corresponding to the smallest prime divisor $p$, must be cyclic and $\theta(G)=\theta(G_p)=\frac{p^2+p-1}{p^3}$. 

If $p=2$, then $\theta(G_2)=5/8$ and by Theorem \ref{5/8-bound-theorem}, $G_2\cong \mathbb{Z}_2\times \mathbb{Z}_2$ or $Q_8$. If $p$ is odd, let $|G_p|=p^n$ and $exp(G_p)=p^k$. Hence by Theorem \ref{exponent-upper-bound}, $$\dfrac{p^2+p-1}{p^3}\leq \dfrac{p^{k-1}(p+1)}{p^{n}}=\dfrac{p+1}{p^{n-k+1}}, \mbox{ i.e., }p^{n-k+1}\leq p^3\cdot \dfrac{p+1}{p^2+p-1}<p^3.$$ Thus $k\leq n\leq k+1$. As $n=k$ implies that $G_p$ is cyclic, we have $n=k+1$. Thus $G_p$ is a non-cyclic group of order $p^{k+1}$ with a cyclic normal subgroup of order $p^k$. So, by Theorem 14.9 (a) (pg. 96, \cite{huppert}), $G_p\cong \mathbb{Z}_{p^k}\times \mathbb{Z}_p$ or $\mathbb{Z}_{p^k}\rtimes \mathbb{Z}_p$ (in case, $k\geq 2$). In both the cases, their cyclic probabilities are same as their order profiles are same and if $k\geq 2$ and $p$ being an odd prime, we have $$\theta(G_p)=\dfrac{1}{p}+\dfrac{p-1}{p^{2k+1}}<\dfrac{1}{p}+\dfrac{p-1}{p^{5}}<\dfrac{p^2+p-1}{p^3}, \mbox{ a contradiction.}$$ Thus $k=1$, i.e., $G_p\cong \mathbb{Z}_p\times \mathbb{Z}_p$. Hence, combining all the cases, we conclude that equality in the above upper bound occurs if and only if $G\cong Q_8\times \mathbb{Z}_m$ (where $m$ is odd) or $(\mathbb{Z}_p\times\mathbb{Z}_p) \times \mathbb{Z}_m$ (where $m$ is such that all prime factors of $m$ are greater than $p$). \qed 


\begin{proposition}\label{G'-prop-1}
    If $G \in \mathcal{G}_p$ (the set of finite groups with $p$ as the smallest prime factor of $|G|$), then $$\theta(G)\leq \dfrac{1}{p^2}\left(1+\dfrac{p^2-1}{|G'|} \right).$$
\end{proposition}
\pf It follows from Theorem 24 \cite{sen} and the fact that $\theta(G)\leq cp(G)$.\qed 

\section{Classification of groups by specific values of $\theta(G)$}\label{specific_values_section}

It is clear that a group $G$ is cyclic if and only if $\theta(G)=1$. However, as with the case of commuting probability, we show that for non-cyclic groups the cyclic probability is bounded away from $1$. Moreover the maximal possible $\theta(G)$ for non-cyclic groups is only a first step. To gain a finer understanding, we investigate which groups attain particular probabilities, such as $5/8,1/2,7/16$ etc. This section provides a complete classification of finite groups whose cyclic probability equals these notable thresholds, revealing structural constraints that often force the group to be a direct product of a small non-cyclic factor with a cyclic group of coprime order.

\begin{lemma}\label{abelian-5/8-equality-lemma}
    If $G$ is a non-cyclic abelian $2$-group and $\theta(G)=5/8$, then $G\cong \mathbb{Z}_2\times \mathbb{Z}_2$.
\end{lemma}
\pf Since $G$ is a finite abelian $2$-group, we can decompose it into a direct product of cyclic $2$-groups:
$$ G \cong \mathbb{Z}_{2^{a_1}} \times \mathbb{Z}_{2^{a_2}} \times \dots \times \mathbb{Z}_{2^{a_k}} $$
where $a_1 \ge a_2 \ge \dots \ge a_k \ge 1$. Since $G$ is non-cyclic, it must have at least two generators, so $k \ge 2$. Note that
$G / \Phi(G) \cong (\mathbb{Z}_2)^k$. By the quotient inequality, $\theta(G) \le \theta(G/\Phi(G))$. In $\mathbb{Z}^k_2$, a pair of elements generates a cyclic group if and only if they are linearly dependent. Out of the $2^{2k}$ total pairs, the number of linearly dependent pairs is $2^k + 2(2^k - 1) = 3 \cdot 2^k - 2$. Thus:
$$\theta((\mathbb{Z}_2)^k) = \frac{3 \cdot 2^k - 2}{2^{2k}} = \frac{3}{2^k} - \frac{2}{4^k}.$$
The function $\frac{3}{2^k} - \frac{2}{4^k}$ is strictly decreasing for $k \ge 2$. Therefore, any group with $k \ge 3$ generators will have a cyclic probability strictly less than $\frac{11}{32}$, which is less than $\frac{5}{8}$. Since we are given $\theta(G) = \frac{5}{8}$, $G$ must be generated by exactly $k=2$ elements. Hence, $G \cong \mathbb{Z}_{2^a} \times \mathbb{Z}_{2^b}$ with $a \ge b \ge 1$. Thus 
$$\dfrac{5}{8}=\theta(G)=\theta(\mathbb{Z}_{2^a} \times \mathbb{Z}_{2^b})=\dfrac{2^{2a+b}+\frac{2}{7}(2^{3b}-1)}{2^{2a+2b}}.$$ It can be easily checked that this holds only when $a =b= 1$, i.e., Therefore, $G \cong \mathbb{Z}_2 \times \mathbb{Z}_2$. \qed

\begin{theorem}\label{5/8-bound-theorem}
Let $G$ be a finite non-cyclic group. Then $\theta(G)\leq \frac{5}{8}$. Moreover, equality holds if and only if $G\cong Q_8\times \mathbb{Z}_m$ or $(\mathbb{Z}_2\times \mathbb{Z}_2)\times \mathbb{Z}_m$ where $m$ is an odd positive integer $\geq 1$.
\end{theorem}
\pf If $G$ is non-abelian, then we have $\theta(G)\leq cp(G)\leq 5/8$. So, we consider only the case when $G$ is an abelian non-cyclic group. 

As $G$ is abelian, it can be expressed as a direct product of its Sylow subgroups and hence $$\theta(G)=\prod_{p\mid |G|}\theta(G_p),$$ where $G_p$ denotes the Sylow $p$-subgroup of $G$. As $G$ is non-cyclic, at least one of the $G_p$ is non-cyclic and $G_p/\Phi(G_p)\cong \mathbb{Z}^k_p$ with $k\geq 2$. So, we have $$\theta(G)\leq \theta(G_p)\leq \theta(G_p/\Phi(G_p))\leq \frac{p^2+p-1}{p^3}\leq \frac{5}{8}.$$

In case of abelian non-cyclic groups, for equality we need every odd Sylow subgroup to be cyclic and the Sylow $2$-subgroup $G_2$ must be non-cyclic abelian with $\theta(G_2)=5/8$. Thus by Lemma \ref{abelian-5/8-equality-lemma}, $G_2\cong \mathbb{Z}_2\times \mathbb{Z}_2$ and $G\cong (\mathbb{Z}_2\times \mathbb{Z}_2)\times \mathbb{Z}_m$, where $m$ is an odd integer $\geq 1$.

In case of non-abelian groups, for equality we must have $\theta(G)=cp(G)=5/8$. From the second equality, it follows that \cite{sen}, $G/Z(G)\cong \mathbb{Z}_2\times \mathbb{Z}_2$. Thus $G$ is nilpotent and $G\cong G_2\times G_{odd}$ where $G_{odd}$ is the direct product of all odd Sylow subgroups of $G$. Moreover, as $G/Z(G)$ is group of order $4$, it follows that $G_{odd}$ is abelian and hence $G_2$ is non-abelian. As $\theta(G)=cp(G)$, every abelian subgroup of $G$ must be cyclic. In particular, $G_2$ is a non-abelian $2$-group whose abelian subgroups are all cyclic and hence by Theorem 4.10 \cite{conrad-generalized-quaternion}, $G_2\cong Q_{2^n}$. As $cp(G)=5/8$, we must have $cp(G_2)=5/8$. The only generalized quaternion group with a commuting probability of $5/8$ is $Q_8$. Again, as $G_{odd}$ is abelian and all its subgroups must be cyclic, $G_{odd}$ must be a cyclic group of odd order, $\mathbb{Z}_m$. Thus $G\cong Q_8\times \mathbb{Z}_m$. \qed

\begin{theorem}\label{CycProb=1/2}
    Let $G$ be a finite group with $\theta(G)=\frac{1}{2}$, then $cp(G)=\frac{1}{2}$ and $G\cong \mathbb{Z}_3\rtimes \mathbb{Z}_{2m}$, where $gcd(m,3)=1$.
\end{theorem}
\pf As $\theta(G)\leq cp(G)$, it suffices to show that $cp(G)\not> \frac{1}{2}$. Suppose, $cp(G)>\frac{1}{2}$, then by \cite{lescot}, $G$ is nilpotent and $G\cong G_2\times G_{odd}$, where $G_2$ is the Sylow $2$-subgroup and $G_{odd}$ is the product of the normal odd Sylow subgroups of $G$. Thus if $|G_2|=2^n$, we have $$\theta(G)=\theta(G_2)\cdot\theta(G_{odd})=\dfrac{|S_2|}{2^{2n}}\cdot \dfrac{|S_{odd}|}{|G_{odd}|^2}=\dfrac{1}{2},$$ i.e., 
\begin{equation}\label{odd-equation}
    |S_2|\cdot |S_{odd}|=2^{2n-1}\cdot |G_{odd}|^2
\end{equation}
{\it Claim 1:} $|S_{odd}|$ is odd.\\
{\it Proof of Claim 1:} For any non-identity element $g\in G_{odd}$, let $o(g)=\Pi_i p^{k_i}_i$ where $p_i$'s are odd primes. Then $\psi(o(g))=\Pi_i p^{k_i-1}_i(p_i+1)$ is even. Thus $$|S_{odd}|=\sum_{g \in G_{odd}} \psi(o(g))=\psi(1)+\sum_{g\neq e}\psi(o(g))=1+even=odd.$$
Thus from Equation \ref{odd-equation}, $|S_2|=2^{2n-1}\cdot m$ where $m$ is an odd integer. However as $|S_2|\leq 2^{2n}$, we have $m=1$ and $|S_2|=2^{2n-1}$.

{\it Claim 2:} $|S_2|\equiv 1\pmod 3$.\\
{\it Proof of Claim 2:} For any non-identity element in $G_2$, its order is $2^j$ for some $j \ge 1$ and $$\psi(2^j) = 2^j\left(1 + \frac{1}{2}\right) = 3 \cdot 2^{j-1} \mbox{ is a multiple of }3.$$ Thus $$|S_2|=\psi(1)+\sum_{g\in G_2\setminus \{e\}} \psi(o(g))=1+3m, \mbox{ where }m\geq 1.$$
 Now, as $2\equiv -1\pmod 3$, we have $|S_2|=2^{2n-1}\equiv (-1)^{2n-1}\equiv -1\pmod 3$, which contradicts Claim 2 and hence $cp(G)=\frac{1}{2}$.
 
 As $cp(G)=1/2$, by Corollary 3.2 \cite{lescot}, $G$ must be isoclinic to $S_3$. Thus $G/Z(G)\cong S_3$ and $G'\cong \mathbb{Z}_3$. Again, as $\theta(G)=cp(G)=1/2$, every abelian subgroup of $G$ must be cyclic. Thus $G$ does not contain any subgroup isomorphic to $\mathbb{Z}_p \times \mathbb{Z}_p$, where $p$ is a prime and hence $G$ contains a unique subgroup of order $p$ for each prime divisor $p$ of $|G|$. Thus all odd Sylow subgroups of $G$ are cyclic and Sylow $2$-subgroup is either cyclic or generalized quaternion.

{\it Claim 3:} Sylow $2$-subgroup of $G$ is cyclic.\\
{\it Proof of Claim 3:} Let $G_2$ be a Sylow $2$-subgroup of $G$. As $G/Z(G)\cong S_3$, we have $[G_2:G_2\cap Z(G)]=2$. Note that $G_2\cap Z(G)\subseteq Z(G_2)$. Thus either $G_2$ is abelian or $G_2/Z(G_2)$ is cyclic of order $2$, thereby making $G_2$ abelian. Hence $G_2$ is not generalized quaternion and $G_2$ is cyclic.

{\it Claim 4:} Sylow $3$-subgroup of $G$ is cyclic of order $3$ and normal in $G$.\\
{\it Proof of Claim 4:} As $G/Z(G)\cong S_3$, we have $|G|=6\cdot |Z(G)|$. If $3\mid |Z(G)|$, then there exists $z\in Z(G)$ such that $o(z)=3$. Let $H=\langle z \rangle$. Thus $G'\cong H\cong \mathbb{Z}_3$. As $G$ contains unique subgroup of order $3$, we must have $G'=H$, i.e., $G'\leq Z(G)$. But this implies $G/Z(G)$ is abelian, a contradiction. Thus $3\nmid |Z(G)|$ and hence $9\nmid |G|$, i.e., Sylow $3$-subgroup of $G$ is cyclic of order $3$. Moreover as $G'\cong \mathbb{Z}_3$, the Sylow $3$-subgroup of $G$ is $G'$ and hence normal in $G$.

Thus by Schur-Zassenahus Theorem, $G\cong \mathbb{Z}_3\rtimes \mathbb{Z}_{2m}$ where $gcd(3,m)=1$. (Note that the semidirect product is non-trivial, as otherwise $G$ would be abelian and $cp(G)=1$.) \qed 

\begin{remark}
     $cp(G)=1/2$ does not imply $\theta(G)=1/2$, e.g., $cp(\mathbb{Z}_3\times S_3)=1/2$, but $\theta(\mathbb{Z}_3\times S_3)=19/54$.
\end{remark}


\begin{theorem}\label{CycProb>1/2}
    If $G$ is non-cyclic and $\theta(G)>1/2$, then $G$ is nilpotent and $G \cong Q_8\times \mathbb{Z}_m$ or $M_{2^k}\times \mathbb{Z}_m$ or $\mathbb{Z}_{2^{k-1}}\times \mathbb{Z}_2 \times \mathbb{Z}_m$, where $m$ is odd and $$\theta(G)=\frac{1}{2}+\frac{1}{2^{2s+1}}, \mbox{ for some }s \in \mathbb{N}.$$
\end{theorem}
\pf It was shown in \cite{machale}, that $cp(G)>1/2$ implies that $G$ is nilpotent and $cp(G)=\frac{1}{2}+\frac{1}{2^{2s+1}}$ for some positive integer $s$. As $1/2<\theta(G)\leq cp(G)$, it follows that $G$ is nilpotent. 


Let $G=G_2\times G_{odd}$ where $G_2$ is the normal Sylow $2$-subgroup of $G$ and $G_{odd}$ is the internal direct product of normal odd Sylow subgroups of $G$. 

{\it Claim:} $G_{odd}$ is cyclic.\\
{\it Proof of Claim:} If $G_{odd}$ is non-cyclic, from the upper bound in Theorem \ref{smallest-prime-theorem}, it follows that $\theta(G_{odd})\leq \frac{3^2+3-1}{3^3}=\frac{11}{27}$. However this implies $$\theta(G)=\theta(G_2)\cdot\theta(G_{odd})\leq 1\cdot \dfrac{11}{27}<\dfrac{1}{2}, \mbox{ a contradiction.}$$ 

As $G_{odd}$ is cyclic and $G$ is non-cyclic, it follows that $G_2$ is non-cyclic. 




Let $exp(G_2)=2^k$. Then $$\dfrac{1}{2}<\theta(G)=\theta(G_2)\leq \dfrac{\psi(2^k)}{|G_2|}=\dfrac{3\cdot 2^{k-1}}{|G_2|},$$ i.e., $|G_2|<3\cdot 2^k$. Again, as $G_2$ is non-cyclic, $|G_2|\geq 2^{k+1}$. Thus $|G_2|=2^{k+1}$, i.e., $G_2$ is a $2$-group of order $2^{k+1}$ with a cyclic maximal subgroup of order $2^k$. Hence by Theorem 14.9(b) (pg. 96, \cite{huppert}), if $k=2$, then $G_2\cong \mathbb{Z}_4\times \mathbb{Z}_2,D_4,Q_8$ and if $k\geq 3$, $G_2$ is isomorphic to either $\mathbb{Z}_{2^{k-1}}\times \mathbb{Z}_2$, or dihedral group $D_{2^{k}}$ or generalized quaternion group $Q_{2^{k+1}}$ or semi-dihedral group $SD(2^{k+1})$ or modular maximal-cyclic group $M_{2^k}$. However, from Table \ref{family}, it follows that only those cyclic probability $>1/2$ are $Q_8,M_{2^k}$ and $\mathbb{Z}_{2^{k-1}}\times \mathbb{Z}_2$ and all of them have cyclic probability as given in the statement of the theorem.


\begin{lemma}\label{no-group-(7/16,1/2)}
    There does not exist any finite group $G$ with $7/16<\theta(G)<1/2$.
\end{lemma}
\pf Let $\mathcal{A}=\{cp(G): G\mbox{ is a finite group}\}$. Then it can be shown that \cite{browning}, $\mathcal{A}\cap [7/16,1/2]=\{7/16,1/2\}$. Let, if possible, $G$ be the smallest group (in terms of order) satisfying the inequality. As $7/16<\theta(G)\leq cp(G)$, we have $cp(G)\geq 1/2$. 

Case 1: $cp(G)>1/2$. In this case, $G$ is nilpotent and $G\cong G_2\times G_{odd}$. If $G_{odd}$ is not cyclic, we have $\theta(G)\leq \theta(G_{odd})\leq \dfrac{11}{27}<\dfrac{7}{16}$, a contradiction. Thus $G_{odd}$ is cyclic and hence $\theta(G)=\theta(G_2)$. As $G$ is the minimum counterexample, we should have $G=G_2$, i.e., $G$ is a $2$-group. Let $exp(G)=2^k$. Then by Theorem \ref{exponent-upper-bound}, we have $$\dfrac{7}{16}<\theta(G)\leq \dfrac{\psi(2^k)}{|G|}.$$ This implies that $$|G|<2^{k+3}\cdot \dfrac{3}{7}<2^{k+3}\cdot \dfrac{1}{2}=2^{k+2}, \mbox{ i.e., }|G|=2^k \mbox{ or }2^{k+1}.$$ As $G$ is non-cyclic and $exp(G)=2^k$, we must have $|G|=2^{k+1}$. Thus $G$ is a non-cyclic group of order $2^{k+1}$ with a cyclic normal subgroup (as $exp(G)=2^k$) of order $2^k$. Hence by Theorem 14.9(b) (pg. 96, \cite{huppert}), if $k=2$, then $G\cong D_4,Q_8$ and if $k\geq 3$, $G$ is isomorphic to either dihedral $D_{2^{k}}$ or generalized quaternion $Q_{2^{k+1}}$ or semi-dihedral $SD(2^{k+1})$ or modular maximal-cyclic group $M_{2^k}$. As $\theta(D_4),\theta(Q_8) \not\in (7/16,1/2)$, they are ruled out. As $cp(D_{2^k})=\frac{1}{4}+\frac{3}{2^{k+1}}<\frac{1}{2}$, it is also ruled out. Similarly other possibilities can also be ruled out using the Table \ref{family}.

Note that if $cp(G)=1/2$ and $G$ is nilpotent, we can rule out the groups as above. 

So, we assume that:\\
Case 2: $cp(G)=1/2$ and $G$ is non-nilpotent: In this case, we have $G/Z(G)\cong S_3$ and $G'\cong \mathbb{Z}_3$. Moreover $cp(G)=1/2$ and $\theta(G)<1/2$. 

{\it Claim 1:} $9\nmid |G|$.\\
{\it Proof of Claim 1:} Let $|G|=2^n\cdot 3^m\cdot t$ where $gcd(t,6)=1$. Then $|Z(G)|=2^{n-1}\cdot 3^{m-1}\cdot t$. If $t>1$, then $G$ has a central subgroup $S$ of order $t$. Thus $7/16<\theta(G)\leq \theta(G/S)$ and hence by minimality of $G$, $\theta(G/S)\geq 1/2$. If $\theta(G/S)> 1/2$, then $G$ is nilpotent, a contradiction. Thus we have $\theta(G/S)=1/2$, i.e., $G/S\cong \mathbb{Z}_3\rtimes \mathbb{Z}_{2l}$ with $gcd(l,3)=1$. Thus $9\nmid |G|$. If $t=1$ and $2\mid |Z(G)|$, then let $S$ be a central subgroup of order $2$. Proceeding similarly, we get $\theta(G/S)=1/2$ and hence $G/S\cong \mathbb{Z}_3\rtimes \mathbb{Z}_{2l}$ with $gcd(l,3)=1$. Hence $9\nmid |G|$. If $t=1$ and $2\nmid |Z(G)|$, then $|Z(G)|=3^{m-1}$. If $|Z(G)|=1$, then $G\cong S_3$ and hence $\theta(G)=1/2$, a contradiction. If $3\mid |Z(G)|$, then we choose $S$ to be a central subgroup of order $3$ and get $\theta(G/S)=1/2$, i.e., $G/S\cong \mathbb{Z}_3\rtimes \mathbb{Z}_{2}$. Thus $|G|=18$ and it can be checked that no such group exist with $7/16<\theta(G)<1/2$. Thus the claim holds.

As $G'\cong \mathbb{Z}_3$, we must have $G_3=G'\cong \mathbb{Z}_3 \lhd G$. Thus $|G|=2^n\cdot 3\cdot t$, where $gcd(t,6)=1$ and $|Z(G)|=2^{n-1}\cdot t$.

{\it Claim 2:} All odd Sylow $p$-subgroups $G_p$ of $G$ are cyclic.\\
{\it Proof of Claim 2:} If $n>1$, take $S$ to be a central subgroup of order $2^{n-1}$. Then $G/S\cong \mathbb{Z}_3\rtimes \mathbb{Z}_{2t}$ which proves the claim. If $n=1$, then $|G|=6t$ and $|Z(G)|=t$. As $t=1$ implies $G\cong S_3$ and $\theta(G)=1/2$ a contradiction, we assume $t>1$. Since $gcd(t,6)=1$, by Schur-Zassenhaus theorem, $G\cong Z(G)\rtimes S_3\cong Z(G)\times S_3$. Thus $\theta(G)=\theta(Z(G))\cdot \theta(S_3)$, i.e., $\theta(Z(G))>7/8$, i.e, $Z(G)$ is cyclic and hence the claim holds.


Also, as $|G/Z(G)|$ has order $6$, all Sylow $p$-subgroups of $G$ with $p>3$ are central and hence normal in $G$.

Consider the map $\pi: G\rightarrow G/Z(G)$ given by $\pi(g)=gZ(G)$. As $\pi$ is a surjective homomorphism, $\pi(G_2)=G_2/Z(G)$ is a Sylow $2$-subgroup of $G/Z(G)\cong S_3$. Thus $|\pi(G_2)|=2$ and $$\pi(G_2)\cong \dfrac{G_2Z(G)}{Z(G)}\cong \dfrac{G_2}{G_2\cap Z(G)}$$  As $C=G_2\cap Z(G)\subseteq Z(G)\subseteq Z(G_2)$, is a central subgroup of index $2$ in $G_2$, $G_2$ must be abelian. As $G$ is non-nilpotent, $G_2$ must be non-normal in $G$.

Thus $G \cong G_{odd}\rtimes G_2$, where $G_{odd}$ is the internal direct product of normal odd Sylow subgroups of $G$. Hence $$\dfrac{7}{16}<\theta(G)\leq \theta(G/G_{odd})=\theta(G_2)\leq \dfrac{\psi(exp(G_2))}{|G_2|}.$$ 
If $exp(G_2)=2^k$, it can be shown, as in Case 1, that $|G_2|=2^k$ or $2^{k+1}$.
As $G_2$ is abelian, we must have $G_2\cong \mathbb{Z}_{2^k}\times \mathbb{Z}_2$ or $\mathbb{Z}_{2^k}$.

Let $H$ be the internal direct product of all Sylow $p$-subgroups for $p>3$ (if any) and $K=G_3\rtimes G_2\cong \mathbb{Z}_3\rtimes_\varphi (\mathbb{Z}_{2^k}\times \mathbb{Z}_2) \mbox{ or } \mathbb{Z}_3\rtimes \mathbb{Z}_{2^k} \leq G$ where $\rtimes_\varphi$ denote a non-trivial semidirect product (otherwise $G$ will be nilpotent). Note that $H$ is a central subgroup of $G$, we have $G\cong H\times K$ with $gcd(|H|,|K|)=1$. Thus, as $H$ is cyclic, by Proposition \ref{quotient-direct-product-prop}, we have 
\begin{equation}\label{eq-2}
  \theta(G)=\theta(H)\cdot \theta(K)=\theta(K).  
\end{equation}

Case 2(a): If $G_2\cong \mathbb{Z}_{2^k}\times \mathbb{Z}_2$, we have $$7/16<\theta(G)=\theta(K)=\theta(\mathbb{Z}_3\rtimes_\varphi (\mathbb{Z}_{2^k}\times \mathbb{Z}_2)).$$

A non-trivial semidirect product $\rtimes_\varphi$ corresponds to a non-trivial homomorphism $\varphi: \mathbb{Z}_{2^k}\times \mathbb{Z}_2\rightarrow Aut(\mathbb{Z}_3)$ and two such products are isomorphic iff the corresponding homomorphisms lie in the same orbit under the action of $Aut(\mathbb{Z}_{2^k}\times \mathbb{Z}_2)$.

For $k=1$, only one such non-isomorphic semidirect product exists and we have $K\cong D_6$, the dihedral group of order $12$. For $k\geq 2$, two such non-isomorphic semidirect products exist and $K\cong \mathbb{Z}_{2^k}\times S_3$ or $K\cong \mathbb{Z}_2\times (\mathbb{Z}_3\rtimes \mathbb{Z}_{2^k})$. As $$\theta(D_6)=\dfrac{3}{8}, \theta(\mathbb{Z}_{2^k}\times S_3)=\dfrac{1}{3}+\dfrac{1}{6\cdot 4^k}, \theta(\mathbb{Z}_2\times (\mathbb{Z}_3\rtimes \mathbb{Z}_{2^k}))=\dfrac{1}{4}+\dfrac{1}{2^{2k+1}}.$$

In all cases, $\theta(K)<7/16$ contradicting that $\theta(G)>7/16$. Thus no such group $G$ exists if $G_2\cong \mathbb{Z}_{2^k}\times \mathbb{Z}_2$.

Case 2(b): If $G_2\cong \mathbb{Z}_{2^k}$, then $K\cong \mathbb{Z}_3\rtimes \mathbb{Z}_{2^k}$, and by Theorem \ref{CycProb=1/2}, we have $\theta(K)=1/2$. However, this along with Equation \ref{eq-2} yields $\theta(G)=1/2$, a contradiction. Hence the theorem follows.\qed

Following similar line of arguments, one can show that 
\begin{enumerate}
    \item If $\theta(G)=7/16$, then $G\cong D_4\times \mathbb{Z}_m$ or $Q_{16}\times \mathbb{Z}_m$ where $m$ is an odd positive integer.
    \item There does not exist any group $G$ with $11/27<\theta(G)<7/16$.
    \item If $\theta(G)=11/27$, then $G$ is abelian and $G\cong \mathbb{Z}_3\times \mathbb{Z}_3\times \mathbb{Z}_m$ where $gcd(m,3)=1$.
\end{enumerate}

We omit those detailed proofs as they are repetitive in nature. Nevertheless we plan to include those proofs in an extended version of this paper in future.

\section{Sufficient Conditions for Nilpotency and (Super)-Solvability}\label{sufficient_conditions_section}
One of the main motivations for studying probabilistic invariants is their power to detect global algebraic properties. While the commuting probability $cp(G)$, $\sigma'(G)$ and $\sigma''(G)$ are known to give criteria for nilpotency and solvability above certain values, we show that $\theta(G)$ can serve an analogous – and sometimes sharper – role. It has already been noted in Theorem \ref{CycProb>1/2} that $\theta(G)>1/2$ implies that $G$ is nilpotent. However, we can make a stronger statement as well, which we prove next.

\begin{theorem}\label{nilpotent-threshold-theorem}
    Let $G$ be a finite group with $\theta(G)>\frac{2}{5}$ and $\theta(G) \neq \frac{1}{2}$, then $G$ is nilpotent.
\end{theorem}
\pf If $G$ is abelian, then the theorem holds. So we assume $G$ to be a non-abelian group such that $\theta(G)>\frac{2}{5}$ and $\theta(G) \neq \frac{1}{2}$. If $p$ is the smallest prime dividing $|G|$, then we have $\frac{2}{5}<\theta(G)\leq \frac{p^2+p-1}{p^3}$, i.e., $p=2$ or $3$. Again, by Proposition \ref{G'-prop-1}, we have $$\frac{2}{5}<\theta(G) \leq \frac{1}{p^2}\left(1+ \frac{p^2-1}{|G'|} \right).$$ 
If $p=3$, then this gives $|G'|\leq 3$, i.e., $G'\cong \mathbb{Z}_3$. However, by Lemma 3.11(iv) \cite{barry}, this implies $G$ is nilpotent. Thus $p=2$ and the above inequality gives $|G'|\leq 4$. Hence $$G' \cong {\mathbb{Z}_2} \text{ or } {\mathbb{Z}_3} \text{ or } {\mathbb{Z}_4} \text{ or } {\mathbb{Z}_2 \times \mathbb{Z}_2}. $$ If $G'\cong \mathbb{Z}_2$, then by Lemma 3.11(iv) \cite{barry}, $G$ is nilpotent. If $G'\cong \mathbb{Z}_2 \times \mathbb{Z}_2$, then by Lemma 3.13 \cite{barry}, $G$ is either nilpotent or $cp(G)=1/3$. The later cannot occur as that would imply $1/3=cp(G)\geq \theta(G)>2/5$, i.e., $1/3>2/5$, a contradiction. So, we are left with cases when $G'\cong \mathbb{Z}_3$ or $\mathbb{Z}_4$.

Let, if possible, $G$ be a non-nilpotent group of minimum order such that $1/2\neq \theta(G)>2/5$. Also we have $G'\cong \mathbb{Z}_3$ or $\mathbb{Z}_4$. If $Z(G)$ is non-trivial, we have $\theta(G/Z(G))\geq \theta(G)>2/5$. Also $G/Z(G)$ is non-nilpotent, this contradicts the minimality of $|G|$, unless $\theta(G/Z(G))=1/2$.  

\noindent {\it Claim:} $\theta(G/Z(G))\neq 1/2$.\\
{\it Proof of Claim:} If $\theta(G/Z(G))=1/2$, then by Theorem \ref{CycProb=1/2}, $cp(G/Z(G))=1/2$ and $G/Z(G) \cong \mathbb{Z}_3\rtimes \mathbb{Z}_{2m}$, where $gcd(m,3)=1$. So, from \cite{rusin}, we get $(G/Z(G))'\cong G'/(G'\cap Z(G))) \cong \mathbb{Z}_3$. As $G'\cong \mathbb{Z}_3$ or $\mathbb{Z}_4$, this rules out $G'\cong \mathbb{Z}_4$ and implies $G'\cong \mathbb{Z}_3$ and $G'\cap Z(G)$ is trivial. Hence by Lemma 3.9 \cite{barry}, we have $|G/Z(G)|\leq 6$, i.e., $G/Z(G)\cong S_3$ (since $G/Z(G)$ is non-nilpotent). Thus $G$ is isoclinic to $S_3$ and hence by Lemma 2.4 \cite{lescot}, $cp(G)=1/2$. Now, proceeding as in Case 2 of proof of Lemma \ref{no-group-(7/16,1/2)}, we get a contradiction.

Thus $Z(G)$ must be trivial and by Lemma 3.9 \cite{barry}, we have $|G|\leq |G'|\cdot |Aut(G')|$. As $|G'|=3$ or $4$, we have $|G|\leq 6$ or $8$. Since $S_3$ is the only non-nilpotent group with order $\leq 8$ and $\theta(S_3)=1/2$, we get a contradiction. Hence the theorem follows.\qed

The bound in the above theorem is tight, as we have a family of non-nilpotent groups $G_k\cong (\mathbb{Z}_5\rtimes \mathbb{Z}_{2^k})$ such that $\theta(G_k)=2/5$. However, in the next theorem, we prove a stronger result.

\begin{theorem}\label{CycProb=2/5_non_nilpotent}
    Let $G$ be a finite group with $\theta(G)=\frac{2}{5}$, then $G$ is non-nilpotent.
\end{theorem}
\pf Let $\theta(G)=\frac{2}{5}$. If possible, let $G$ be nilpotent. Then $$G \cong \prod_{p\mid |G|}{G_{p}} \mbox{ i.e., } \theta(G)=\prod_{p\mid |G|}{\theta(G_{p})}.$$  Since $\theta(G)=\frac{2}{5}$, we have $5\mid|G|$ and $P_5$ must be non-cyclic. Thus we have $$\frac{2}{5}= \theta(G) \leq \theta(G_5)\leq \frac{29}{125}, \mbox{ a contradiction.}$$ \qed 

\begin{theorem}\label{supersolvable_threshold_theorem}
    If $\theta(G)>7/24$, then $G$ is supersolvable.
\end{theorem}
\pf If $G$ is abelian, there is nothing to prove. Let $G$ be non-abelian and $p$ be the smallest prime dividing $|G|$. If $p\geq 5$, then by Theorem \ref{smallest-prime-theorem}, $\theta(G)\leq 29/125<7/24$, a contradiction. Thus the smallest prime dividing $|G|$ is $2$ or $3$, i.e., $G \in \mathcal{G}_2$ or $\mathcal{G}_3$, where $\mathcal{G}_p$ denote the set of finite groups with $p$ as the smallest prime factor of $|G|$. 

If $G \in \mathcal{G}_3$, then by Proposition \ref{G'-prop-1} we have $$\dfrac{7}{24}<\theta(G)\leq \dfrac{1}{9}\left(1+\dfrac{8}{|G'|} \right), \mbox{ i.e., }|G'|<5.$$ If $G'$ is cyclic, then $G$ is supersolvable. Thus the only case which survives with $|G'|<5$ is $G'\cong \mathbb{Z}_2\times \mathbb{Z}_2$. However, as $G\in \mathcal{G}_3$, i.e., $3$ is the smallest prime dividing $|G|$, we have $|G'|\neq 4$.

So, we assume that  $G \in \mathcal{G}_2$ is the minimum counterexample, i.e, $G$ is a non-supersolvable group of minimum order such that $\theta(G)>7/24$. By Proposition \ref{G'-prop-1}, we have $$\dfrac{7}{24}<\theta(G)\leq \dfrac{1}{4}\left(1+\dfrac{3}{|G'|} \right), \mbox{ i.e., }|G'|<18.$$

As groups with cyclic commutator subgroup are supersolvable, $G \in \mathcal{G}_2$ and by Theorem 4.1 \cite{heffernan}, the only groups which are still left as a potential candidate of $G'$ are: $\mathbb{Z}_2\times \mathbb{Z}_2,\mathbb{Z}_4\times \mathbb{Z}_2,  \mathbb{Z}^3_2, \mathbb{Z}_6\times \mathbb{Z}_2, \mathbb{Z}_4\times \mathbb{Z}_4,  \mathbb{Z}_8\times \mathbb{Z}_2, \mathbb{Z}_4\times \mathbb{Z}_2\times \mathbb{Z}_2, \mathbb{Z}^4_2$.

Moreover, if $Z(G)$ is non-trivial, then $G$ has a central subgroup $N$ of prime order and $G/N$ is non-supersolvable, i.e., $\frac{7}{24}<\theta(G)\leq \theta(G/N)$, which contradicts the minimality of $G$ as a counterexample. Thus $Z(G)$ is trivial and hence, using Lemma 3.9 \cite{barry}, we have $|G|\leq |G'|\cdot |Aut(G')|$.

Thus for all the above choices of $G'$, we run an exhaustive search using GAP/SAGE \cite{sage} on non-supersolvable groups $G$ such that 
\begin{itemize}
    \item $|G|\leq |G'|\cdot |Aut(G')|$.
    \item $|Z(G)|=1$.
    \item $G$ has no cyclic normal subgroups.
    \item $\theta(G)>7/24$
\end{itemize}
Except for $G'\cong \mathbb{Z}_4\times \mathbb{Z}_2\times \mathbb{Z}_2$ and $\mathbb{Z}^4_2$, for all other choices of $G'$, it was seen that no such groups exist. The bound $|G|\leq |G'|\cdot |Aut(G')|$ is too large for the above two groups to run an exhaustive search using the Small Group library in GAP (See Table \ref{G-order-bound}). So, we treat them separately.

\begin{table}[h]
    \centering
    \begin{tabular}{|c|c||c|c|}
     \hline   $G'$ & $|G'|\cdot |Aut(G')|$ & $G'$ & $|G'|\cdot |Aut(G')|$ \\ \hline
       $\mathbb{Z}_2\times \mathbb{Z}_2$  & $24$ & $\mathbb{Z}_4\times \mathbb{Z}_4$  & $1536$\\ \hline
         $\mathbb{Z}_4\times \mathbb{Z}_2$  & $64$ & $\mathbb{Z}_8\times \mathbb{Z}_2$  & $256$\\ \hline
         $\mathbb{Z}^3_2$  & $1344$ & $\mathbb{Z}_4\times \mathbb{Z}_2\times \mathbb{Z}_2$  & $3072$\\ \hline
         $\mathbb{Z}_6\times \mathbb{Z}_2$  & $144$ & $\mathbb{Z}^4_2$  & $322560$\\ \hline
    \end{tabular}
    \caption{Upper Bounds on $|G|$}
    \label{G-order-bound}
\end{table}

Note that $G'\cong \mathbb{Z}_4\times \mathbb{Z}_2\times \mathbb{Z}_2$ and $Z(G)$ is trivial cannot hold simultaneously as $G'$ has a unique element $(2,0,0)$ of order $2$, which is a square in $G'$ and hence central in $G$. So we are left only with the case when $G\cong \mathbb{Z}^4_2$.

If $G'\cong \mathbb{Z}^4_2$, let $x\in G$ and $h=xG'\in G/G'=A$ (say) and $x^{o(h)}\in G'$. As $G'$ is elementary abelian $2$-group, $x^{2\cdot o(h)}=e$, i.e., $o(x)$ divides $2\cdot o(h)$. Using submultiplicativity of $\psi$, we get $\psi(o(x))\leq 3\cdot \psi(o(h))$.  

As each coset $h$ contains exactly $16$ elements of $G$, its total contribution to $|S_G|$ is at most $16\cdot 3\cdot \psi(o(h))$. Summing over all $h \in A$, we get $$|S_G|\leq 48\sum_{h \in G/G'}\psi(o(h))=48\cdot |S_A|.$$
Thus $$\theta(G)=\dfrac{|S_G|}{|G|^2}\leq \dfrac{48\cdot |S_A|}{16^2\cdot |A|^2}=\dfrac{3}{16}\cdot \theta(A)\leq \dfrac{3}{16}<\dfrac{7}{24}, \mbox{ a contradiction}.$$\qed

\begin{theorem}\label{solvable_threshold_theorem}
    If $\theta(G)>3/40$, then $G$ is solvable.
\end{theorem}
\pf As $cp(G)\geq \theta(G)>3/40$, by Theorem 11 \cite{guralnick}, either $G$ is solvable, or else $G\cong A_5\times T$ for some abelian group $T$. So, it is enough to rule out the second case. However, in that case, we have $$\dfrac{3}{40}<\theta(G)=\theta(A_5\times T)\leq \theta(A_5)\cdot \theta(T)\leq \dfrac{3}{40},$$ a contradiction. Thus $G$ is solvable.

\begin{remark}
    All the bounds given in Theorems \ref{nilpotent-threshold-theorem}, \ref{supersolvable_threshold_theorem} and \ref{solvable_threshold_theorem} are tight, as $\theta(D_5\times \mathbb{Z}_3)=2/5$, $\theta(A_4)=7/24$ and $\theta(A_5)=3/40$.
\end{remark}

Theorems \ref{nilpotent-threshold-theorem}, \ref{supersolvable_threshold_theorem} and \ref{solvable_threshold_theorem} provide sufficient conditions for nilpotency, supersolvability and solvability of groups respectively. As in literature, many such sufficient conditions exist, it is natural to investigate their respective strengths. In view of that, we emphasize, in the next table, why $\theta(G)$ is better than some of the existing sufficient conditions while determining solvability/nilpotency of certain groups. The \textcolor{blue}{blue} ones are where the indicator works and \textcolor{red}{red} ones are where it does not. It is evident from Table \ref{comparison_table} that there are instances of groups where the other three indicators fail to recognize the group property, but $\theta(G)$ can detect it. Thus we conclude that this new parameter performs better than the other three, in many cases, if not in all cases.
\begin{table}[h]
    \centering
    \begin{tabular}{|c|c|c|c|c|c|}
      \hline Property & Group $G$ & $\theta(G)$ & $cp(G)$ & $\sigma'(G)$ & $\sigma''(G)$\\ \hline \hline 
      & & & & & \\
       Supersolvability & $D_{15}$ & \textcolor{blue}{$\dfrac{3}{10}>\dfrac{7}{24}$} & \textcolor{red}{$\dfrac{3}{10}<\dfrac{1}{3}$} & \textcolor{red}{$\dfrac{59}{147}<\dfrac{31}{77}$} & \textcolor{red}{$\dfrac{59}{300}<\dfrac{31}{144}$}\\  
       & & & & & \\ \hline
        & & & & & \\ 
      Solvability & $\mathbb{Z}_{19}\rtimes \mathbb{Z}_6$ & \textcolor{blue}{$\dfrac{3}{38}>\dfrac{3}{40}$} &  \textcolor{blue}{$\dfrac{3}{38}>\dfrac{3}{40}$} & \textcolor{red}{$\dfrac{241}{2401}<\dfrac{211}{1617}$} & \textcolor{red}{$\dfrac{241}{4332}<\dfrac{211}{3600}$}\\   & & & & & \\ \hline
       &  & & & & \\ 
       Nilpotency & $\mathbb{Z}_4\times \mathbb{Z}_2$ & \textcolor{blue}{$\dfrac{17}{32}>\dfrac{1}{2}$} & \textcolor{blue}{$1>\dfrac{1}{2}$} & \textcolor{red}{$\dfrac{23}{43}<\dfrac{13}{21}$} & \textcolor{red}{$\dfrac{23}{64}<\dfrac{13}{36}$} \\ 
        & & & & & \\ \hline
    \end{tabular}
    \caption{Comparison of $\theta(G),cp(G),\sigma'(G)$ and $\sigma''(G)$ as indicators}
    \label{comparison_table}
\end{table}

\section{Conclusion and Open Issues}\label{conclusion_section}
In this paper we have systematically investigated the cyclic probability $\theta(G)$ of a finite group $G$, which measures the likelihood that two randomly chosen elements generate a cyclic subgroup. Our results establish $\theta(G)$ as a powerful probabilistic invariant that captures subtle structural properties of finite groups, often outperforming classical invariants such as the commuting probability $cp(G)$ and the sum-of-element-orders based invariants $\sigma'(G)$ and $\sigma''(G)$.

Despite the progress made, several natural questions remain open for future investigation:
\begin{enumerate}
    \item \textbf{Classification for other rational values:} While we have classified groups for $\theta(G) = 5/8,\,1/2,\,7/16,\,11/27$, the complete set of possible values of $\theta(G)$ for finite groups is unknown. Is the set of cyclic probabilities dense in some interval? Are there isolated accumulation points analogous to those known for $cp(G)$?

    \item \textbf{Computational exploration:} The GAP‑based searches in this paper were limited by group order. A more extensive computational study of $\theta(G)$ for groups of order up to, say, $2000$ or $5000$ could reveal new patterns, counterexamples, or conjectures regarding the distribution and thresholds of cyclic probabilities.

    \item \textbf{Graph‑theoretic connections:} The cyclic probability is closely related to the enhanced power graph of a group, where the vertices are the elements of the group and the edges connect elements that generate a cyclic subgroup. Exploring this connection further — for instance, relating $\theta(G)$ to the edge density or clique number of the enhanced power graph — may yield new structural insights.

\end{enumerate}

\section*{Acknowledgement}
The first author is supported by the CSIR PhD fellowship (File no. $08/0155(26311)/2026-EMR-I$), Govt. of India. The second author acknowledges the funding of DST-FIST Sanction no. $SR/FST/MS-I/2019/41$, Govt. of India. 

\subsection*{Statements and Declarations}
Data sharing is not applicable to this article, as no datasets were generated or analyzed during the current study. The authors have no competing interests to declare that are relevant to the content of this article.

\end{document}